\documentclass[12pt,reqno]{amsart}
\usepackage{bbm}
\usepackage{mathrsfs}
\usepackage{amsfonts}
\usepackage{latexsym,amsmath,amsfonts,amssymb,amsthm}
 \theoremstyle{plain}
\newtheorem{Thm}{Theorem}
\newtheorem{Lem}{Lemma}

\newtheorem{Conj}{Conjecture}

\theoremstyle{definition}
\newtheorem*{Ack}{Acknowledgment}
\theoremstyle{remark}
\newtheorem*{Rem}{Remark}

\def\r{{\mathbf r}}
\def\cH{{\mathcal{H}}}
\def\cR{{\mathcal{R}}}

\def\N{\mathbb N}

\def\S{\mathcal{S}}
\def\T{\mathcal{T}}

\def\cA{\mathcal{A}}

\def\sP{\mathscr{P}}

\def\1{{\bf 1}}
\def\pmod #1{\ ({\rm mod}\ #1)}

\def\floor #1{\lfloor{#1}\rfloor}

\begin{document}

\title{On the number of distinct prime factors of $nj+a^hk$}
\author{Hao Pan}
\email{haopan79@yahoo.com.cn}
\address{Department of Mathematics, Nanjing University,
Nanjing 210093, People's Republic of China} \subjclass[2000]{Primary
11P32; Secondary 11A07, 11B05, 11B25, 11N36}\keywords{prime factor,
power of 2}
\thanks{The author is supported by
the National Natural Science Foundation of China (Grant No.
10771135).} \maketitle
\begin{abstract}
Let $\omega(n)$ denote the number of distinct prime factors of $n$.
Then for any given $K\geq 2$, small $\epsilon>0$ and sufficiently
large (only depending on $K$ and $\epsilon$) $x$, there exist at
least $x^{1-\epsilon}$ integers $n\in[x,(1+K^{-1})x]$ such that
$\omega(nj\pm a^hk)\geq(\log\log\log x)^{\frac{1}{3}-\epsilon}$ for
all $2\leq a\leq K$, $1\leq j,k\leq K$ and $0\leq h\leq K\log x$.
\end{abstract}

\section{Introduction}
\setcounter{equation}{0} \setcounter{Thm}{0} \setcounter{Lem}{0}
\setcounter{Cor}{0}

In 1849, Polignac conjectured that every odd integer greater than 3
is the sum of a prime and a power of 2. However, 127 is an evident
counterexample for Polignac's conjecture. In 1950, van der Corput
\cite{vanderCorput50} proved that there are a positive proportion of
positive odd integers not of the form $p+2^h$ with $p$ is prime and
$h\in\N$. In fact, using covering congruences, Erd\H os
\cite{Erdos50} constructed a residue class of odd integers, which
contains no integers of the form $p+2^h$.

In 1975, using Erd\H os' idea, Cohen and Selfridge
\cite{CohenSelfridge75} found a residue class of odd integers, whose
every elements can not be representable as $\pm p^\alpha\pm 2^h$.
And another example with a smaller modulus may be found in
\cite{Sun00}. Recently, using Selberg's sieve method, Tao \cite{Tao}
proved that for any give integer $K\geq 2$, there exist at least
$c_Kx/\log x$ primes $p$ in the interval $[x, (1+K^{-1})x]$
satisfying $|pj\pm a^hk|$ is composite for every $2\leq a\leq K$,
$1\leq j,k\leq K$ and $1\leq h\leq K\log x$, where $c_K$ is a
constant only depending on $K$.

Let $\omega(n)$ denote the number of distinct prime factors of $n$.
In \cite{Erdos77}, Erd\H os mentioned that if there exist
incongruent covering systems with arbitrarily large least modulus,
then for any integer $K\geq 2$, the set
$$
\{n\geq 1:\,n\text{ is odd and is not of the form }q+2^h\text{ with
}\omega(q)\leq K\text{ and }h\in\N\}
$$
contains an infinite arithmetic progression. By modifying Tao's
discussions, in this paper, we shall prove that
\begin{Thm}\label{t1} Suppose that $K\geq 2$ is an integer and $\epsilon>0$ is small number. Then for sufficiently large (only depending on $K$ and
$\epsilon$) $x$, there exist at least $x^{1-\epsilon}$ integers
$n\in[x,(1+K^{-1})x]$ such that $\omega(nj\pm a^hk)\geq(\log\log\log
x)^{\frac{1}{3}-\epsilon}$ for all $2\leq a\leq K$,\ $1\leq j,k\leq
K$ and $0\leq h\leq K\log x$.
\end{Thm}
The proof of Theorem \ref{t1} will be given in the next section. And
unless indicated otherwise, the constants implied by $\ll$, $\gg$
and $O(\cdot)$ only depend on $K$ and $\epsilon$.

\section{Proof of Theorem \ref{t1}}
\setcounter{equation}{0} \setcounter{Thm}{0} \setcounter{Lem}{0}
\setcounter{Cor}{0}

The following lemma is the well-known Mertens theorem.
\begin{Lem}[{\cite[Theorems 6.6, 6.7, 6.8]{Nathanson96}}]
\label{mertens}
$$
\sum_{\substack{p\leq x\\ p\text{ prime}}}\frac{\log p}{p}=\log
x+c_1+O\bigg(\frac{1}{\log x}\bigg),
$$
$$
\sum_{\substack{p\leq x\\ p\text{ prime}}}\frac{1}{p}=\log\log
x+c_2+O\bigg(\frac{1}{\log x}\bigg),
$$
and
$$
\prod_{\substack{p\leq x\\ p\text{
prime}}}\bigg(1-\frac{1}{p}\bigg)^{-1}=e^{\gamma}\log x+O(1),
$$
where $\gamma$ is the Euler constant and $c_1,c_2$ are absolute
constants.

\end{Lem}
Define
$$
\sP(z)=\prod_{\substack{p<z\\ p\text{ prime}}}p.
$$
\begin{Lem}
\label{GzWz} Suppose that $A, B>100$ and $z>e^{100A}$. Suppose that
$\omega$ is a function satisfying
$$
0\leq\omega(p)\leq\min\{A,(1-B^{-1})p\}
$$
for any prime $p$. Then
$$
\frac{1}{G(z)}\leq A^{2AB}W(z),
$$
where
$$
G(z)=\sum_{\substack{d\mid \sP(z)\\
d<z}}\prod_{p\mid d}\frac{\omega(p)}{p-\omega(p)}
$$
and
$$
W(z)=\prod_{\substack{p<z\\
p\text{ prime}}}\bigg(1-\frac{\omega(p)}{p}\bigg).
$$
\end{Lem}
\begin{proof} Define
$$
G(\xi,z)=\sum_{\substack{d\mid \sP(z)\\ d<\xi}}\prod_{p\mid
d}\frac{\omega(p)}{p-\omega(p)}
$$
for $\xi\geq z$. Then by the discussions in the proof of \cite[Lemma
4.1]{HalberstamRichert74}, for any $\sigma\leq 1$, we have
\begin{align*}
1-W(z)G(\xi,z) \leq\exp\bigg(-(1-\sigma)\log\xi+\sum_{\substack{
\substack{p<z\\
p\text{
prime}}}}\bigg(\frac{1}{p^\sigma}-\frac1p\bigg)\omega(p)\bigg).
\end{align*}
Letting $\sigma=1-1/\log z$, we get
\begin{align*}
\sum_{\substack{ p<z\\
p\text{
prime}}}\bigg(\frac{1}{p^\sigma}-\frac1p\bigg)\omega(p)=\sum_{\substack{
p<z\\ p\text{ prime}}}\frac{\omega(p)}{p}(e^{\frac{\log p}{\log
z}}-1)\leq \frac{e-1}{\log z}\sum_{\substack{ p<z\\
p\text{ prime}}}\frac{\omega(p)\log p}{p}\leq eA.
\end{align*}
Hence
\begin{equation}
\label{wzgz} 1-W(z)G(\xi,z)\leq\exp\bigg(-\frac{\log\xi}{\log
z}+eA\bigg). \end{equation} Letting $C=eA+1$, by (\ref{wzgz}), we
have
$$
W(z^{\frac1C})G(z,z^{\frac1C})\geq1-e^{-C+eA}\geq1/2.
$$
And
$$
\log W(z^{\frac1C})-\log W(z)=-\sum_{\substack{ z^{\frac1C}\leq p<z\\
p\text{
prime}}}\log(1-\omega(p)/p)\leq B\sum_{\substack{ z^{\frac1C}\leq p<z\\
p\text{ prime}}}\frac{\omega(p)}p\leq \frac32AB\log C.
$$
Thus
$$
W(z)G(z,z)\geq\frac{W(z)}{W(z^\frac1C)}\cdot
W(z^{\frac1C})G(z,z^{\frac1C})\geq \frac{1}{2}C^{-\frac32AB}.
$$
Hence by (\ref{wzgz}).
$$
\frac{1}{G(z,z)}\leq
W(z)\bigg(1+\frac1{W(z)G(z,z)}\exp(-1+eA)\bigg)\leq A^{2AB}W(z).
$$
\end{proof}

Let $\phi$ denote the Euler totient function.
\begin{Lem}
\label{selberg} Suppose that $W,b\geq 1$ with $(W,b)=1$. Suppose
that $x\geq 1$ and $1\leq z\leq x^{\frac13}$. Then
$$
|\{1\leq n\leq x:\,(Wn+b,\sP(z))=1\}|\ll\frac{x}{\log
z}\cdot\frac{W}{\phi(W)}.
$$
In particular,
$$
|\{1\leq n\leq x:\,Wn+b\text{ is prime}\}|\ll\frac{x}{\log
x}\cdot\frac{W}{\phi(W)}.
$$
\end{Lem}
\begin{proof} This is a simple application of the Selberg sieve
method (cf. \cite[Theorems 3.2 and 4.1]{HalberstamRichert74}).
\end{proof}
\begin{Lem}\label{omegak} Suppose that $x\geq 1$ and $1\leq k\leq\log\log x$. Suppose that
$W,b\geq 1$ and $(W,b)=1$. If $W\leq x^{\frac1{2k}}$ and $\log x\leq
x^{\frac{1}{6k}}$, then
$$
|\{1\leq n\leq x:\, \omega(Wn+b)=k\}|\leq\frac{C^kx(\log\log
x)^{k-1}}{\log x}\cdot\frac{W}{\phi(W)},
$$
where $C>0$ is a constant.
\end{Lem}
\begin{proof}
We use induction on $k$. The case $k=1$ easily follows from Lemma
\ref{selberg}. Suppose that $k\geq 2$. Then,
\begin{align*}
&|\{1\leq n\leq x:\, \omega(Wn+b)=k\}|\\
\leq&\sum_{\substack{1\leq p\leq x^{\frac1{3k}},\ p\nmid W\\
\alpha\geq 1,\ p^\alpha\leq x^{\frac1{k}}}}|\{1\leq
n\leq x/p^{\alpha}:\,\omega(Wn+b')=k-1,\ \text{where }b'\text{ satisfies }b'p^\alpha\equiv b\pmod{W}\}|\\
&+\sum_{\substack{p\leq x^{\frac1{3k}},\ p\nmid W\\ \alpha\geq 1,\
p^\alpha\geq x^{\frac1{k}}}}|\{1\leq n\leq x:\,
Wn+b\equiv0\pmod{p^\alpha},\ ((Wn+b)/p^\alpha,p)=1
\}|\\
&+|\{1\leq n\leq x:\,(Wn+b,\sP(x^{\frac1{3k}}))=1\}|.
\end{align*}
By Lemma \ref{selberg},
$$
|\{1\leq n\leq x:\, (Wn+b,\sP(x^{\frac1{3k}}))=1\}|\leq
\frac{c_1x}{\log(x^{\frac1{3k}})}\prod_{p\mid
W}\bigg(1-\frac{1}{p}\bigg)^{-1}
$$
for some constant $c_1>0$. And
\begin{align*}
&\sum_{\substack{p\leq x^{\frac1{3k}},\ p\nmid W\\\alpha\geq 1,\
p^\alpha\geq x^{\frac1{k}}}}|\{1\leq n\leq x:\, Wn+b\equiv
0\pmod{p^\alpha},\ ((Wn+b)/p^\alpha,p)=1
\}|\\
\leq&2\sum_{\substack{p\leq x^{\frac1{3k}}\\\alpha\geq 3,\
x^{\frac1{k}}\leq p^\alpha\leq x
}}\frac{x}{p^\alpha}+\sum_{\substack{p\leq
x^{\frac1{3k}}}}1=O\bigg(x^{1-\frac{1}{k}}\cdot\frac{x^{\frac1{3k}}}{\log(x^{\frac1{3k}})}\cdot\log
x\bigg)\leq c_2kx^{1-\frac{2}{3k}}
\end{align*}
for some $c_2>0$. Further, by Lemma \ref{mertens}, for
$$
\sum_{\substack{p\leq x^{\frac1{3k}}\\ p\text{
prime}}}\frac{1}{p-1}\leq c_3\log\log x
$$
for some $c_3>0$. Choose $C\geq2(c_1+c_2+c_3)$. Notice that for
$p^\alpha<x^{\frac1{k}}$,
$$
x/{p^\alpha}\geq x^{\frac{k-1}{k}}\geq \max\{W^{2(k-1)}, (\log
x)^{6(k-1)}\}
$$
and
$$
k-1\leq\log\log
x-1=\log\log(x^{\frac{1}{e}})\leq\log\log(x/p^\alpha).
$$
So by the induction hypothesis,
$$
|\{1\leq n\leq
x/p^\alpha:\,\omega(Wn+b')=k-1\}|\leq\frac{C^{k-1}x(\log\log(x/p^\alpha))^{k-2}}{p^\alpha\log(x/p^\alpha)}\cdot\frac{W}{\phi(W)}.
$$
Hence,
\begin{align*}
&|\{1\leq n\leq x:\, \omega(Wn+b)=k\}|\\
\leq&\frac{2C^{k-1}Wx(\log\log x)^{k-2}}{\phi(W)\log
x}\sum_{\substack{p\leq x^{\frac1{2k}}\\ p\text{
prime}}}\frac{1}{p-1}+\frac{2c_1kx}{\phi(W)\log
x}+c_2x^{1-\frac{2}{3k}}\log\log x\\
\leq&\frac{C^kx(\log\log x)^{k-1}}{\log x}\cdot\frac{W}{\phi(W)}.
\end{align*}
\end{proof}

Now suppose that $x$ is sufficiently large. Let $L=\lfloor
(\log\log\log x)^{\frac{1}{3}-\epsilon}\rfloor+1$ and
$Q=\exp((\log\log x)^{1-\epsilon})$. Clearly $L\ll(\log\log
 Q)^{\frac{1}{3}-\epsilon}$.
Let
$$
\cR=\{(a,j,k,l):\,2\leq a\leq K,\ 1\leq j,|k|\leq K,\ 1\leq l\leq
L\}
$$
and $M=\floor{16K^2}+1$. Clearly $M$ is a constant only depending on
$K$, and $ML\geq16K^2L=8|\cR|$.

Below we shall choose some distinct primes
$$p_{\r,t},\ \r\in\cR,\ 1\leq t\leq T_{\r}$$
in the interval $[\exp((\log Q)^{\frac{4}{ML}}),Q]$ satisfying that
\begin{equation}
\label{productp} \frac{1}{2}(\log
Q)^{\frac{1}{ML}}\leq\prod_{t=1}^{T_{\r}}\bigg(1-\frac{1}{p_{\r,t}}\bigg)^{-1}\leq\frac{3}{2}(\log
Q)^{\frac{1}{ML}} \end{equation} for any fixed $\r\in\cR$. And
assume that we have chosen primes $p_{\r,t}, 1\leq t\leq T_{\r}$ for
some $\r=(a,j,k,l)\in\cR$. Then for any
$I\subseteq\{1,\ldots,T_\r\}$, let $$ m_{\r,I}=\prod_{t\in
I}p_{\r,t}$$ and $q_{\r,I}$ be the largest primitive prime factor of
$a^{m_{\r,I}}-1$, i.e., $q_{\r,I}\mid a^{m_{\r,I}}-1$ but
$q_{\r,I}\nmid a^{m}-1$ for any $1\leq m<m_{\r,I}$. In particular,
we set $m_{\r,\emptyset}=1$.

First, let $p_{\r,t},\ \r=(2,j,k,l)\in\cR,\ 1\leq t\leq T_\r$ be
distinct primes in the interval $[\exp((\log
Q)^{\frac{4}{ML}}),\exp((\log Q)^{\frac{4+4K^2L}{ML}})]$ satisfying
that
$$
\frac{1}{2}(\log
Q)^{\frac{1}{ML}}\leq\prod_{t=1}^{T_{\r}}\bigg(1-\frac{1}{p_{\r,t}}\bigg)^{-1}\leq\frac{3}{2}(\log
Q)^{\frac{1}{ML}}.
$$
Suppose that $a>2$ and we have chosen distinct prime $p_{\r',t},
1\leq t\leq T_{\r'}$ in the interval $[\exp((\log
Q)^{\frac{4}{ML}}),\exp((\log Q)^{\frac{4+8(a-1)K^2L}{ML}})]$ for
every $\r'=(a',j,k,l)\in\cR$ with $2\leq a'<a$. Let
$$
w_{a}=\prod_{\substack{\r'=(a',j,k,l)\in\cR\\
2\leq a'<a}}\prod_{\substack{I\subseteq\{1,\ldots,T_{\r'}\}\\ 1\leq
|I|\leq 2ML^2}}(q_{\r',I}-1).
$$
Clearly,
\begin{align*}
\frac{\log(w_a)}{\log a}\leq&\sum_{\substack{\r'=(a',j,k,l)\in\cR\\
2\leq a'<a}}\sum_{\substack{I\subseteq\{1,\ldots,T_{\r'}\}\\ 1\leq
|I|\leq
2ML^2}}m_{\r',I}\leq\bigg(\sum_{\substack{\r'=(a',j,k,l)\in\cR\\
2\leq a'<a}}\sum_{t=1}^{T_{\r'}}p_{\r',t}\bigg)^{2ML^2}\\
\leq&(\exp((\log
Q)^{\frac{4+8(a-1)K^2L}{ML}})^2)^{2ML^2}=\exp(4ML^2(\log
Q)^{\frac{4+8(a-1)K^2L}{ML}}).
\end{align*}
Thus we get
$$
\omega(w_{a})\leq\frac{\log(w_a)}{\log 2}\leq\exp({5ML^2}(\log
Q)^{\frac{4+8(a-1)K^2L}{ML}})\leq\exp((\log
Q)^{\frac{5+8(a-1)K^2L}{ML}}),
$$
by noting that
\begin{align*}
\log(5ML^2(\log
Q)^{\frac{4+8(a-1)K^2L}{ML}})=&\log(5ML^2)+\frac{4+8(a-1)K^2L}{ML}\log\log
Q\\
\leq&\frac{5+8(a-1)K^2L}{ML}\log\log Q. \end{align*} Furthermore, by
the prime number theorem, there exist
$$
(1+o(1))\bigg(\frac{\exp((\log Q)^{\frac{6+8(a-1)K^2L}{ML}})}{(\log
Q)^{\frac{6+8(a-1)K^2L}{ML}}}-\frac{\exp((\log
Q)^{\frac{5+8(a-1)K^2L}{ML}})}{(\log
Q)^{\frac{5+8(a-1)K^2L}{ML}}}\bigg)\geq\exp((\log
Q)^{\frac{5+8(a-1)K^2L}{ML}})
$$
primes in the interval $[\exp((\log
Q)^{\frac{5+8(a-1)K^2L}{ML}}),\exp((\log
Q)^{\frac{6+8(a-1)K^2L}{ML}})]$, by noting that clearly $\exp((\log
Q)^{\frac{1}{ML}})\geq 4\log\log Q$. Notice that
$$
\prod_{\substack{\exp((\log Q)^{\frac{5+8(a-1)K^2L}{ML}})\leq p\leq
\exp((\log
Q)^{\frac{6+8(a-1)K^2L}{ML}})\\
p\text{ prime}}}\bigg(1-\frac{1}{p}\bigg)^{-1}\leq(\log
Q)^{\frac{6+8(a-1)K^2L}{ML}}
$$
and
$$
\prod_{\substack{\exp((\log Q)^{\frac{5+8(a-1)K^2L}{ML}})\leq p\leq
\exp((\log
Q)^{\frac{4+8aK^2L}{ML}})\\
p\text{ prime}}}\bigg(1-\frac{1}{p}\bigg)^{-1}\geq\frac{1}{2}(\log
Q)^{\frac{4+8aK^2L}{ML}}.
$$
Hence,
$$
\prod_{\substack{\exp((\log Q)^{\frac{5+8(a-1)K^2L}{ML}})\leq p\leq
\exp((\log
Q)^{\frac{4+8aK^2L}{ML}})\\
p\text{ prime and }p\nmid
w_{a}}}\bigg(1-\frac{1}{p}\bigg)^{-1}\geq\frac{1}{2}(\log
Q)^{\frac{8K^2L-2}{ML}}\geq (\log Q)^{\frac{4K^2L}{ML}}.
$$
Thus we may choose distinct primes $p_{\r,t},\ \r=(a,j,k,l)\in\cR,\
1\leq t\leq T_\r$ in the interval $[\exp((\log
Q)^{\frac{5+8(a-1)K^2L}{ML}}),\exp((\log Q)^{\frac{4+8aK^2L}{ML}})]$
satisfying that $p_{\r,t}\nmid w_a$ and
$$
\frac{1}{2}(\log
Q)^{\frac{1}{ML}}\leq\prod_{t=1}^{T_{\r}}\bigg(1-\frac{1}{p_{\r,t}}\bigg)^{-1}\leq\frac{3}{2}(\log
Q)^{\frac{1}{ML}}.
$$
Repeat this process from $a=3$ to $K$, until we complete the choices
of $p_{\r,t},\ 1\leq t\leq T_\r$ for all $\r\in\cR$.

Since $p_{\r,t}\nmid w_a$ for any $\r=(a,j,k,l)\in\cR$ and
$p_{\r,t}\mid q_{\r,I}-1$ for any $t\in
I\subseteq\{1,\ldots,T_{\r}\}$ with $|I|\leq 2ML^2$, we have
$q_{r,I}\not=q_{\r',I'}$ for every $\r'=(a',j,k,l)\in\cR$ with
$a'<a$ and $I'\subseteq\{1,\ldots,T_{\r'}\}$ with with $1\leq
|I'|\leq 2ML^2$. That is, all these $q_{\r,I}$ are distinct. And
since there are
$$
(1+o(1))\bigg(\frac{\exp((\log
Q)^{\frac{4}{ML}}\cdot\frac{1}{2}(\log Q)^{\frac{1}{ML}})}{(\log
Q)^{\frac{4}{ML}}\cdot\frac{1}{2}(\log
Q)^{\frac{1}{ML}}}-\frac{\exp((\log Q)^{\frac{4}{ML}})}{(\log
Q)^{\frac{4}{ML}}}\bigg)
$$
primes in $[\exp((\log Q)^{\frac{4}{ML}}),\exp(\frac12(\log
Q)^{\frac{5}{ML}})]$, clearly we have $T_{\r}\geq (2ML^2)^2$.

For each $\r=(a,j,k,l)\in\cR$, let
$$
W_\r=\prod_{\substack{I\subseteq\{1,\ldots,T_\r\}\\ 1\leq |I|\leq
2ML^2}}q_{\r,I}
$$
and let $b_\r$ be an integer such that
$$
b_\r j+a^{|I|-1}k\equiv0\pmod{q_{\r,I}}
$$
for every $I\subseteq\{1,\ldots,T_r\}$ with $1\leq|I|\leq 2ML^2$.
Let
$$
W=\prod_{\r\in\cR}W_\r
$$
and let $b$ be an integer such that
$$
b\equiv b_\r\pmod{W_\r}
$$
for every $r\in\cR$. Then
$$
W\leq\prod_{\r\in\cR}\prod_{\substack{I\subseteq\{1,\ldots,T_\r\}\\
1\leq |I|\leq
2ML^2}}(K^{m_{\r,I}}-1)<K^{\sum_{\r\in\cR}(1+p_{\r,1}+\ldots+p_{\r,T_\r})^{2ML^2}}\leq
K^{2K^3LQ^{4ML^2}}.
$$
Since
$$
\log\log W\ll\log(2K^2L)+4ML^2\log Q\ll(\log\log\log
x)^{\frac{2}{3}-2\epsilon}\cdot(\log\log x)^{1-\epsilon},
$$
we have $W\leq x^{\frac{\epsilon}2}$ provided that $x$ is
sufficiently large. Let
$$
\S=\{x\leq n\leq (1+K^{-1})x:\,n\equiv b\pmod{W}\}
$$
and
$$
\T=\{n\in\S:\,\omega(nj+a^hk)<L\text{ for some }2\leq a\leq K,\
1\leq j,|k|\leq K,\ 0\leq h\leq K\log x\}.
$$
For any $\r\in\cR$, let $\cH_\r$ be the set
$$
\{0\leq h\leq K\log x:\,h\not\equiv |I|-1\pmod{m_{\r,I}}\text{ for
any }I\subseteq\{1,\ldots,T_\r\}\text{ with }1\leq |I|\leq 2ML^2\}.
$$
The following lemma is the key of our proof.
\begin{Lem}
\label{ur}
$$
|\cH_\r|\ll \frac{(16ML^2)^{8M^2L^4}(\log\log Q)^{4M^2L^4}}{(\log
Q)^{2L}}\cdot K\log x.
$$
\end{Lem}
\begin{proof}
Suppose that $h\in\cH_\r$. Clearly, by the pigeonhole principle and
the definition of $\cH_\r$, we have
$$
|\{t\in\{1,\ldots,T_\r\}:\,\{h\}_{p_{\r,t}}< 2ML^2\}|<(2ML^2)^2,
$$
where $\{h\}_p$ denotes the least non-negative reside of $h$ modulo
$p$. Therefore
$$
\cH_\r\subseteq\bigcup_{\substack{J\subseteq\{1,\ldots,T_\r\}\\
0\leq |J|<(2ML^2)^2\\ c\in \mathcal{C}_J}}\{1\leq h\leq K\log x:\,
h\equiv c\pmod{m_{\r,J}},\ (\prod_{0\leq
s\leq2ML^2-1}(h-s),m_{J}^{*})=1\},
$$
where
$$
m_{J}^{*}=\frac{1}{m_{\r,J}}\prod_{t=1}^{T_\r}p_{\r,t}
$$
and
$$
\mathcal{C}_J=\{0\leq c<m_{\r,J}:\, \{c\}_{p_t}< 2ML^2\text{ for all
}t\in J\}.
$$

For any $J\subseteq\{1,\ldots,T_\r\}$ with $0\leq
|J|\leq(2ML^2)^2-1$ and $c\in\mathcal{C}_J$, let
$$
\cA_{J,c}=\{\prod_{0\leq s\leq2ML^2-1}(m_{\r,J}d+c-s):\,0\leq d\leq
K\log x/m_{\r,J}\}.
$$
Since
$$
\log(Q^{(2ML^2)^2})=4M^2L^4\log Q\ll(\log\log
x)^{1-\frac{\epsilon}2},
$$
we have $m_{\r,J}\leq Q^{(2ML^2)^2}\leq (\log x)^{\frac18}$. Let
$z=(K\log x)^{\frac18}$. Applying Selberg's sieve method, we have
$$
|\{u\in\cA_{J,c}:\, (u,m_{J}^{*})=1\}|\leq
\frac{|\cA_{J,c}|}{G(z)}+\sum_{\substack{d\mid m_J^*\\
d<z^2}}3^{\omega(d)}|r_d|,
$$
where
$$
G(z)=\sum_{\substack{d\mid m_J^*\\
d<z}}\prod_{p\mid d}\frac{2ML^2}{p-2ML^2}
$$
and
$$
r_d=|\{u\in\cA_{J,c}:\, u\equiv 0\pmod{d}\}|-|\cA_{J,c}|\prod_{p\mid
d}\frac{2ML^2}{p}.
$$
By Lemma \ref{GzWz},
$$
\frac{1}{G(z)}\ll (2ML^2)^{8ML^2}\prod_{p\mid m_J^*}\bigg(1-\frac
{2ML^2}p\bigg)\leq(2ML^2)^{8ML^2}\prod_{p\mid m_J^*}\bigg(1-\frac
{1}p\bigg)^{2ML^2}.
$$
Since $|r_d|\ll(2ML^2)^{\omega(d)}$ and $\omega(d)\ll\log d/\log\log
d$, we have
$$
\sum_{\substack{d\mid m_J^*\\
d<z^2}}3^{\omega(d)}|r_d|\ll(2ML^2)^{O(\frac{\log z}{\log\log
z})}z^{3}.
$$
And noting that $2ML^2\ll(\log\log\log x)^{1-\epsilon}\ll (\log\log
z)^{1-\epsilon}$, we get $(2ML^2)^{O(\frac{\log z}{\log\log z})}\ll
z$.

Thus since $|\mathcal{C}_J|\leq (2ML^2)^{(2ML^2)^2}$,
\begin{align*}
|\cH_\r|\ll&(2ML^2)^{(2ML^2)^2}\sum_{\substack{J\subseteq\{1,\ldots,T_\r\}\\
0\leq |J|\leq (2ML^2)^2-1}}(2ML^2)^{800ML^2}\frac{K\log
x}{m_{\r,J}}\prod_{p\mid
m_J^*}\bigg(1-\frac {1}p\bigg)^{2ML^2}\\
\leq&(2ML^2)^{8M^2L^4}\bigg(1+\sum_{t=1}^{T_\r}\frac{1}{p_{\r,t}-1}\bigg)^{(2ML^2)^2}\prod_{t=1}^{T_\r}\bigg(1-\frac {1}{p_{\r,t}}\bigg)^{2ML^2}\cdot K\log x\\
\ll&\frac{(16ML^2)^{8M^2L^4}(\log\log Q)^{4M^2L^4}}{(\log
Q)^{2L}}\cdot K\log x.
\end{align*}

\end{proof}

Let $\cH=\bigcup_{\r\in\cR}\cH_\r$. In view of Lemma \ref{ur},
\begin{align}
|\cH|\ll&|\cR|\cdot\frac{(16ML^2)^{8M^2L^4}(\log\log
Q)^{4M^2L^4}}{(\log Q)^{2L}}\cdot K\log
x\notag\\
\leq&\frac{2K^4L(32ML^2)^{8M^2L^4}(\log\log\log x)^{4M^2L^4}\log
x}{(\log\log x)^{2(1-\epsilon)L}}.
\end{align}
Suppose that $n\in\T$, i.e., there exist $2\leq a\leq K$, $1\leq
j,|k|\leq K$ and $0\leq h\leq K\log x$ such that
$\omega(nj+a^hk)<L$. We claim that $h\in\cH$. In fact, assume on the
contrary that $h\not\in\cH$. Then for any $1\leq l\leq L$, letting
$\r_l=(a,j,k,l)$, there exists $I_l\subseteq\{1,\ldots,T_{\r_l}\}$
with $1\leq |I_l|\leq 2ML^2$ such that
$h\equiv|I_l|-1\pmod{m_{\r_l,I_l}}$. Recalling that $n\equiv
b\pmod{q_{\r_l,I_l}}$ and $q_{\r_l,I_l}\mid a^{m_{\r_l,I_l}}-1$, we
have
$$
nj+a^hk\equiv bj+a^{|I_l|-1}k\equiv0\pmod{q_{\r_l,I_l}}.
$$
It follows that
$$
nj+a^hk\equiv0\pmod{\prod_{1\leq l\leq L}q_{\r_l,I_l}},
$$
and $\omega(nj+a^hk)\geq L$.

Thus we get
$$
\T\subseteq\bigcup_{\substack{\ 1\leq j,|k|\leq K\\
h\in\cH\\ 2\leq a\leq K}}\{x\leq n\leq(1+K^{-1})x:\,\ n\equiv
b\pmod{W},\ \omega(nj+a^hk)<L\}.
$$
Notice that
$$
\log\log ((WK)^{4ML^2})\leq\log(4ML^2)+2\log\log W\ll(\log\log
x)^{1-\frac\epsilon2}.
$$
For fixed $j,k,a,h$, letting $g=(Wj,bj+a^hk)$, by Lemma
\ref{omegak}, we have
\begin{align*}
&|\{x\leq n\leq(1+K^{-1})x:\,\ n\equiv b\pmod{W},\
\omega(nj+a^hk)<L\}|\\
\leq&|\{(x-b)/W\leq n\leq((1+K^{-1})x-b)/W:\,
\omega(Wjn/g+(bj+a^hk)/g)<L\}|\\
\leq&\frac{LC^{L-2}(K^{-1}x/W)(\log\log(K^{-1}x/W))^{L-2}}{\log(K^{-1}x/W)}\cdot\frac{Wj/g}{\phi(Wj/g)}
\leq\frac{2LK^{-1}C^{L-2}xj(\log\log x)^{L-2}}{\phi(Wj)\log x}.
\end{align*}
Therefore
\begin{align*}
|\T|\ll&2K^3\cdot|\cH|\cdot\frac{LC^{L-1}x(\log\log
x)^{L-2}}{\phi(W)\log x}\prod_{\substack{p\leq K\\ p\text{ prime}}}\bigg(1-\frac1p\bigg)^{-1}\\
\ll&\frac{4K^8L(32ML^2)^{8M^2L^4}(\log\log\log x)^{4M^2L^4}\log
x}{(\log\log x)^{2(1-\epsilon)L}}\cdot\frac{LC^{L-1}x(\log\log
x)^{L-2}}{\phi(W)\log x}\\
\ll&\frac{4K^8(32CM\log\log\log x)^{10M^2L^4}x}{W(\log\log
x)^{(1-2\epsilon)L}},
\end{align*}
where the last inequality follows from $W/\phi(W)\ll\log\log W$.
Noting that
$$
\log((32CM\log\log\log x)^{10M^2L^4})\ll M^2L^4\log\log\log\log x\ll
L(\log\log\log x)^{1-\epsilon},
$$
we have
$$
\lim_{x\to+\infty}\frac{(32CM\log\log\log x)^{10M^2L^4}}{(\log\log
x)^{(1-2\epsilon)L}}=0.
$$
It follows that
$$
|\T|\leq \frac{x}{4KW}
$$
provided that $x$ is sufficiently large.

Finally,
\begin{align*}
&|\S\setminus\T|=|\S|-|\T|\geq\frac{x}{KW}-1-\frac{x}{4KW}\gg
x^{1-\frac\epsilon2},
\end{align*}
i.e.,
\begin{align*}
&\{1\leq n\leq x:\,\omega(nj+a^hk)\geq L\text{ for all }2\leq a\leq
K,\ 1\leq j,|k|\leq K,\ 0\leq h\leq K\log x\}
\end{align*}
has at least $x^{1-\epsilon}$ elements for sufficiently large $x$.

\begin{Rem}
Since Tur\'an had proved that $\omega(n)=(1+o(1))\log\log n$ for
almost all integers $n$, we believe that the result of Theorem
\ref{t1} is far from satisfaction. We have the following conjecture.
\begin{Conj}
For any given large $K>0$, small $\epsilon>0$ and sufficiently large
(only depending on $K$ and $\epsilon$) $x$, there exist at least
$x^{1-\epsilon}$ integers $n\in[x,(1+K^{-1})x]$ such that
$\omega(nj\pm a^hk)\geq(\log\log x)^{1-\epsilon}$ for all $2\leq
a\leq K$, $1\leq j,k\leq K$ and $0\leq h\leq K\log x$.
\end{Conj}
\end{Rem}

\begin{Ack} The author thanks Professor Emmanuel Vantieghem for pointing out an error on the history of Polignac's conjecture in the earlier version. The author also thanks Professor Zhi-Wei Sun for his
helpful discussions.
\end{Ack}

\end{document}